\documentclass{article}
\usepackage[numbers, square]{natbib}

\usepackage[a4paper, total={6in, 8in}]{geometry}
\usepackage{listings}
\usepackage{setspace}
\usepackage{hyperref}
\usepackage{acro}
\usepackage{amsmath}
\usepackage{amsthm}
\usepackage{amssymb}
\usepackage{graphicx}
\usepackage{geometry}

\usepackage[T1]{fontenc}
\usepackage[utf8]{inputenc}
\usepackage{mathtools}
\usepackage{subcaption}
\usepackage{cancel}
\usepackage{mathtools}   
\usepackage{amsfonts}
\usepackage{tikz}
\usepackage{color}

\makeatletter
\renewcommand{\@biblabel}[1]{[#1]}
\makeatother

\title{On the group of $\omega^{k}$-preserving diffeomorphisms}

\author{Habib Alizadeh}

\date{} 

\newtheorem{theorem}{Theorem}
\newtheorem{corollary}[theorem]{Corollary}
\newtheorem{lemma}[theorem]{Lemma}

\newtheorem{remark}[theorem]{Remark}
\newtheorem{definition}[theorem]{Definition}
\newtheorem{def/theorem}[theorem]{Theorem/Definition}

\newtheorem{conjecture}[theorem]{Conjecture}

\newcommand{\C}{\mathbb{C}}
\newcommand{\R}{\mathbb{R}}

\newcommand{\w}{\omega}

\begin{document}

\maketitle

\begin{abstract}
We show that if a diffeomorphism of a symplectic manifold $(M^{2n},\omega)$ preserves the form $\omega^{k}$ for $0 < k < n$ and is connected to identity through such diffeomorphisms then it is indeed a symplectomorphism.
\end{abstract}

\vspace{0.5cm}
\begin{center}
1. INTRODUCTION
\end{center}

A symplectic manifold is a smooth manifold equipped with a closed non-degenerate $2$-form $\w$ called a symplectic form. The dimension of such manifold must be even. Darboux's classical theorem says that symplectic manifolds have no interesting local properties, namely, any symplectic manifold $M$ of dimension $2n$ with a symplectic form $\w$ is locally symplectomorphic to $(\R^{2n}, \w_{0})$ where $\w_{0}$ is the standard symplectic form $dx_{1}\wedge dy_{1} + \dots + dx_{n}\wedge dy_{n}$. A symplectic form defines a certain signed area for surfaces inside the symplectic manifold. It is locally the sum of the areas of projections of the surface into the planes $\langle x_{i},y_{i}\rangle, i=1,\dots,n$, where $\{x_{1},y_{1},\dots, x_{n},y_{n}\}$ are Darboux coordinates. If the dimension of a symplectic manifold is $2$, the symplectic form is just the standard area form in Darboux local coordinates.

In a symplectic manifold $(M^{2n},\w)$ the form $\w^{n}$ is a volume form on $M$ as $\w$ is a non-degenerate $2$-form. If a diffeomorphism preserves the symplectic form it obviously preserves the volume form $\w^{n}$. In \cite{gromovoriginal} Gromov developed his theory of pseudo-holomorphic curves in symplectic manifolds and used it to prove his celebrated non-squeezing theorem. The non-squeezing theorem tells that if the ball $B^{2n}(r)$ of radius $r$ and centered at the origin in $(\R^{2n},\w_{0})$ is symplectically embedded in the cylinder $B^{2}(R)\times \R^{2(n-1)}$ then we must have $r \leq R$. This rigidity phenomenon shows that in fact the two groups, $\text{Symp}(M,\w)$, the group of $\w$-preserving diffeomorphisms, and $\text{Diff}(M,\w^{n})$, the group of volume-preserving diffeomorphisms, are very different. However, there are not any intermediate groups between $\text{Symp}_{0}(M,\w)$, the identity component of $\text{Symp}(M,\w)$, and $\text{Diff}_{0}(M,\w^{n})$, the identity component of $\text{Diff}(M,\w^{n})$. Obviously $\text{Symp}_{0}(M,\w)$ is a subgroup of $\text{Diff}_{0}(M,\w^{n})$. In his book \cite[p.346]{partial} Gromov proves the so called Maximality Theorem: if $(M,\w)$ is a closed connected symplectic manifold and $G$ is a connected subgroup of $\text{Diff}_{0}(M,\w^{n})$ containing $\text{Ham}(M,\w)$, the group of Hamiltonian symplectomorphisms, and an element $\psi$ so that $\psi^{*}\w \neq \pm \w$ then it also contains $\text{Diff}_{0}^{e}(M,\w^{n})$, the group of exact volume-preserving diffeomorphisms. In particular it is equal to $\text{Diff}_{0}(M,\w^{n})$ if $H^{1}(M,\R) = 0$. Consequently if $H^{1}(M,\R) = 0$ and $G$ is a connected group lying in between $\text{Symp}_{0}(M,\w)$ and $\text{Diff}_{0}(M,\w^{n})$ then it is either equal to $\text{Symp}_{0}(M,\w)$ or $\text{Diff}_{0}(M,\w^{n})$. The smoothness of the volume-preserving maps is very important here, i.e. the analogue of the Maximality Theorem is not true when we replace the group $\text{Diff}_{0}(M,\w^{n})$ with the group of volume-preserving homeomorphisms $\text{Homeo}(M,\w^{n})$.

Considering the group $\text{Symp}(M,\w)$ as a subgroup of $\text{Homeo}(M)$, the group of homeomorphisms, and denoting by $\text{Sympeo}(M,\w)$ the closure of $\text{Symp}(M,\w)$ inside $\text{Homeo}(M)$ with $C^{0}$-topology, a celebrated theorem due to Gromov and Eliashberg, see \cite{partial} and \cite{gromoveliashberg}, asserts that if $\phi \in \text{Sympeo}(M,\w)$ is smooth then it is indeed a symplectomorphism. Associated to the theorem of Gromov and Eliashberg there is the $C^{0}$-flux conjecture. 
\begin{conjecture}
For a closed connected symplectic manifold $(M,\w)$ the group $\text{Ham}(M,\w)$ of Hamiltonian diffeomorphisms is $C^{0}$-closed in $\text{Symp}_{0}(M,\w)$. 
\end{conjecture}
Elements of $\text{Sympeo}(M,\w)$ are called symplectic homeomorphisms. A consequence of the Gromov-Eliashberg theorem is that the group $\text{Sympeo}(M,\w)$ is a proper subgroup of the group of volume-preserving homeomorphisms $Homeo(M,\w^{n})$. In \cite{flexibleexample} Buhovsky and Opshtein constructed an example of a symplectic homeomorphism of the the standard $\C^{3}$ whose restriction to the symplectic subspace $\C \times 0\times 0$ is the contraction $(z,0,0)\mapsto (\frac{1}{2}z, 0, 0)$ which is impossible for a symplectic diffeomorphism. Also in \cite{seyfeddini} the autors showed that symplectic homeomorphisms are more rigid than just volume preserving maps by showing that if the image of a coisotropic submanifold is smooth under a symplectic homeomorphism then the image is also coisotropic. Thus the group $\text{Sympeo}(M,\w)$ lies in between $\text{Symp}(M,\w)$ and $\text{Homeo}(M)$ but not equal to neither of them. Getting back to the smooth case and considering $\text{Symp}_{0}(M,\w)$ as a subgroup of $\text{Diff}_{0}(M,\w^{n})$ a very natural example of an intermediate subgroup $G$ would be $\text{Diff}_{0}(M,\w^{k})$ for $0 < k < n$, the identity component of the group of $\w^{k}$-preserving diffeomorphisms. In this paper we explore which side the group $\text{Diff}_{0}(M,\w^{k})$ is equal to. We will prove

\begin{theorem}[Main theorem]
\label{main theorem}
For every symplectic manifold $(M,\w)$ of dimension $2n$ with $n > 2$ we have $\text{Diff}_{0}(M,\w^{k}) = \text{Symp}_{0}(M,\w)$ for all $0 < k < n$.
\end{theorem}

Note that for $n = 2$ the statement is tautological. We explore three different proofs. The first proof uses the Maximality theorem of Gromov, which directly implies the main theorem for closed connected symplectic manifolds up to giving an example of an exact volume preserving diffeomorphism that does not preserve the form $\omega^{k}$ for $0 < k < n$. The second and the third proof are more elementary where we first reduce the problem to a linear algebra problem which we solve using two different methods: one uses the K\"ahler identities and the other is an inductive argument.

\vspace{0.5cm}
\begin{center}
ACKNOWLEDGEMENTS
\end{center}

This research is part of my PhD research program at the Universit\'e de Montreal under the supervision of Egor Shelukhin. I would like to thank him for his comprehensive guidance and many fruitful discussions. I also thank him for pointing out the Maximality Theorem and suggesting to use it to prove the main theorem and teaching me the standard trick in representation theory used in Lemma \ref{span}. I wish to thank Fran\c cois Lalonde for his encouraging enthusiasm about the project and sharing some possible prospects of this project. At the end, I am grateful to Filip Brocic and Marcelo Atallah useful discussions. This research was partially supported by Fondation Courtois.

\vspace{0.5cm}
\begin{center}
2. PRELIMINARIES
\end{center}

In this section we first recall a few relevant definitions, we then state the Maximality theorem of Gromov and present its proof for the convenience of the reader.
\begin{definition}
A smooth manifold $M$ is called a symplectic manifold if it is equipped with a closed non-degenerate $2$-form $\w$, and the form $\w$ is called the symplectic form of $M$. The group of all diffeomorphisms of $M$ that preserve its symplectic form $\w$ is denoted by $Symp(M,\w)$ and its identity component by $Symp_{0}(M,\w)$. 
\end{definition}

\begin{definition}
Let $(M,\w)$ be a symplectic manifold. A diffeomorphism $\phi:M\rightarrow M$ is called Hamiltonian if there is a smooth map $H: M\times [0,1]\rightarrow \R$ so that $\phi = f^{H}_{1}$ where $f^{H}_{t}:M\rightarrow M$ is the flow of the vector field $X_{t}$ that is defined by $\iota_{X_{t}}\w = -dH_{t}$. The group of all Hamiltonian diffeomorphisms of $(M,\w)$ is called the Hamiltonian group and is denoted by $\text{Ham}(M,\w)$.
\end{definition}

\begin{definition}
Let $(M,\w)$ be a symplectic manifold and $1\leq k\leq n$. The group of all diffeomorphisms of $M$ that preserve the $2k$-form $\w^{k}$ is denoted by $\text{Diff}(M,\w^{k})$ and its identity component by $G_{k}:= \text{Diff}_{0}(M,\w^{k})$. In case $k = n$ we also define $\text{Diff}_{0}^{e}(M,\w^{n})$ to be the group of all $\w^{n}$-preserving (volume preserving) diffeomorphisms $f:M\rightarrow M$ so that there is a smooth map $f:M\times [0,1]\rightarrow M$ with $f_{t}$ being a diffeomorphism for all $t$, $f_{0} \equiv id$,  $f_{1}\equiv f$ and $[\iota_{X_{t}}\w^{n}] = 0\in H^{2n-1}_{dR}(M,\R)$ where $X_{t} = (\frac{d}{dt}f_{t})\circ f_{t}^{-1}$. 
\end{definition}

For the convenience of the reader we will state the Maximality theorem of Gromov and present his proof of the theorem below. Then as a consequence we will present the first proof of the main theorem (Theorem \ref{main theorem}).

\begin{theorem}[Maximality Theorem]\cite[p.346]{partial}
\label{maximality theorem}
Let $(M^{2n},\w)$ be a closed connected symplectic manifold. Let $G \subset G_{n}$ be a subgroup that contains the Hamiltonian group $\text{Ham}(M,\w)$ and an element $\psi$ so that $\psi^{*}\w \not\equiv \pm \w$. If $H^{1}(M,\R) = 0$ then $G = G_{n}$. Moreover if $H^{1}(M,\R)\neq 0$ then $\text{Diff}_{0}^{e}(M,\w^{n})\subset G$.
\end{theorem}

Before we start the proof of the Maximality theorem let us prove a linear algebraic lemma which will be used during the proof.

\begin{lemma}
\label{span}
Let $(V,\w)$ be an $2n$-dimensional vector space equipped with a symplectic bilinear form $\w$. If there is a non-degenerate skew-symmetric bilinear form $\alpha : V\times V\rightarrow \R$ so that $\alpha \not\equiv \lambda \w$ for every $\lambda \in \R$ then the orbit of $\alpha$ under the set 
$$Sp(V):= \{ T:V\rightarrow V :\ T^{*}\w = \w \}$$
together with $\w$ itself generates the space of $2$-forms on $V$, i.e.
$${\bigwedge}^{2}V = \big<T^{*}\alpha\ :\ T\in Sp(V) \big> + \langle\w\rangle.$$
\end{lemma}

\begin{proof}
By choosing a a symplectic basis for $(V,\w)$ we assume that $V = \R^{2n}$ and $\w = \w_{0}$, the standard symplectic form on $\R^{2n}$. Let $G = Sp(2n)$, the group of symplectic linear transformations, and $W = \bigwedge^{2}\R^{2n}$. Consider the representation $\rho : G \rightarrow GL(W)$ of $G$ given by $g \mapsto (w\mapsto g^{*}w)$. Suppose $\alpha \in W$ is a non-degenerate $2$-form independent of $\w_{0}$, i.e. $\alpha \not\equiv \lambda\w_{0}$ for any  $\lambda \in \R$. Let $U$ be the subspace of $W$ generated by the orbit of $\alpha$ under the action of $G$, i.e. $U = \langle \rho(g)\alpha : g\in G\rangle$. Then $U + \langle\w_{0}\rangle$ is a $G$-subrepresentation of $W$ and we need to prove that $U + \langle\w_{0}\rangle = W$. Let $T$ be the subgroup of $G$ consisting of the following symplectic diagonal matrices:
\begin{center}
\[
\begin{bmatrix}
t_{1} & 0 & \dots & & 0\\
0 & t_{1}^{-1} & \dots & & 0\\
\vdots &\vdots & \ddots & & \vdots\\
 &  &  & t_{n} & 0\\
0 & 0 & \dots & 0 & t_{n}^{-1}\\
\end{bmatrix}, \ \ \ t_{i} \in \R^{*}, \ i = 1,\dots, n
\]
\end{center}
The space $W$ is a semisimple $T$-representation which decomposes as follows 
$$W = \bigoplus_{1\leq i < j\leq n}\big(E_{ij}\oplus E_{ij}'\big) \bigoplus_{1\leq i\neq j\leq n}F_{ij}  \bigoplus_{1\leq i\leq n} F_{i}$$
where,
$$E_{ij} = \text{span}(dx_{i}\wedge dx_{j}), \ \ E_{ij}' = \text{span}(dy_{i}\wedge dy_{j})$$
$$F_{ij} = \text{span}(dx_{i}\wedge dy_{j}), \ \ F_{i} = \text{span}(dx_{i}\wedge dy_{i})$$
are the weight spaces of the $W$ as a representation of $T$. Since $U + \langle\w_{0}\rangle$ is a $T$-subrepresentation of $W$ and $W$ is a direct sum of multiple of pairwise non-isomorphic irreducible $T$-representations, i.e. semisimple, thus $U + \langle\w_{0}\rangle$ is also direct sum of some of the components of $W$. We show that if $U + \langle\w_{0}\rangle$ contains one of the components of $W$ then $U + \langle\w_{0}\rangle$ contains all of them. We prove this through the following ("loop of") steps: 
\begin{itemize}
\item $F_{r}\subset U + \langle\w_{0}\rangle\implies \oplus_{i,j}F_{ij}\subset U + \langle\w_{0}\rangle$ : Use the symplectic transformation that interchanges the $x_{r}y_{r}$-plane with $x_{i}y_{i}$-plane to see that $\oplus_{i}F_{i}\subset U + \langle\w_{0}\rangle$. To prove inclusion of $F_{ij}$ consider the following symplectic linear transformation $f_{i,j}$,
$$(x,y)\mapsto (x_{1},\dots, x_{i}, \dots, x_{j} - x_{i}, \dots, x_{n}, y_{1},\dots, y_{i} + y_{j},\dots, y_{j}, \dots, y_{n}).$$
Then we have $dx_{i}\wedge dy_{j} = f_{i,j}^{*}(dx_{i}\wedge dy_{i}) - dx_{i}\wedge dy_{i} \in U + \langle\w_{0}\rangle$.

\item $F_{rs} \subset U + \langle\w_{0}\rangle \implies \oplus_{i,j}E_{ij} \oplus_{i,j} E_{ij}' \subset U + \langle\w_{0}\rangle$ : By interchanging $x_{r}y_{r}$-plane with $x_{i}y_{i}$-plane and $x_{s}y_{s}$-plane with $x_{j}y_{j}$-plane we have $\oplus_{i,j}F_{ij}\subset U + \langle\w_{0}\rangle$. To prove inclusion of $E_{ij}$, use the symplectic transformation $(x_{j},y_{j})\mapsto (-y_{j},x_{j})$ on $F_{ij}$ where it is identity on the rest of the coordinates. The analogous argument works for $E_{ij}'$.
\item $E_{rs}$ or $E_{rs}' \subset U + \langle\w_{0}\rangle\implies \oplus_{i}F_{i}\subset U + \langle\w_{0}\rangle$ : If $E_{rs}\subset U + \langle\w_{0}\rangle$ then clearly $E_{rs}'\subset U + \langle\w_{0}\rangle$ and vice-versa. So assume both inclusions of $E_{rs}$ and $E_{rs}'$. Consider the following transformation $f$:
$$(x,y)\mapsto (x_{1},\dots, x_{r} + y_{s},\dots, x_{s} + y_{r},\dots, x_{n},y_{1},\dots, y_{n}).$$
Then we have,
$$dx_{r}\wedge dy_{r} - dx_{s}\wedge dy_{s} = f^{*}(dx_{r}\wedge dx_{s}) - dx_{r}\wedge dx_{s} + dy_{r}\wedge dy_{s}.$$
Therefore for every $1\leq i\leq n$ we have $\alpha_{i} := dx_{r}\wedge dy_{r} - dx_{i}\wedge dy_{i}\in U + \langle\w_{0}\rangle$. So we get,
$$dx_{r}\wedge dy_{r} = \frac{1}{n}(\alpha_{1} + \dots + \alpha_{n} + \w_{0}) \in U + \langle\w_{0}\rangle.$$
By interchanging the $x_{r}y_{r}$-plane with $x_{i}y_{i}$-planes for $1\leq i\leq n$ we get $\oplus_{i}F_{i}\subset U + \langle\w_{0}\rangle$.
\end{itemize}
\end{proof}

\begin{proof}[Proof of Theorem \ref{maximality theorem}]
Let $\phi_{t}\in \text{Diff}_{0}^{e}(M,\w^{n}),\ t\in [0,1]$ and $\phi_{0}=id$. Let $X_{t} = \partial_{t}\phi_{t}\circ \phi_{t}^{-1}$. Then we have that $\iota_{X_{t}}\w^{n} = d\eta_{t}$ for some $\eta_{t}\in \Omega^{2n-2}(M)$. Suppose that $\w_{1},\dots,\w_{N}$ are symplectic forms so that
$${\bigwedge}^{2n-2}_{x}T^{*}M = \big<\{(\w^{n-1}_{1})_{x},\dots, (\w^{n-1}_{N})_{x}\}\big>,\ \ \ \forall x\in M$$
and for all $i,\ \ \w_{i}^{n} = \w^{n}$, such a family of forms $\{\w_{1},\dots, \w_{N}\}$ is called a large family and we will prove their existence in Lemma \ref{large family} below. Thus, one can choose $\eta_{t}$ and $H_{i}\in C^{\infty}(M\times [0,1])$ so that,
$$\eta_{t} = \sum_{i=1}^{N}H_{i}\w_{i}^{n-1}\implies \iota_{X_{t}}\w^{n} = \sum_{i=1}^{N}dH_{i}\w_{i}^{n-1}.$$ 
Define the vector fields $X_{i}$ by $\iota_{X_{i}}\w_{i} = n\ dH_{i}$ for $i = 1,\dots, N$. Then $X_{t} = \sum_{i=1}^{N}X_{i}$.

\begin{remark}
We can choose $\eta$ and consequently $X_{i}$'s canonically using Hodge-deRham Theory.
\end{remark}

\begin{lemma}
\label{large family}
There exist $f_{1},\dots,f_{N}\in G$ such that the symplectic forms $\{\w_{i} := f_{i}^{*}\w\}_{i=1}^{N}$ form a large family, i.e. they satisfy the following:
\begin{itemize}
\item{$\w_{i}^{n} = \w^{n},\ i\in \{1,\dots,N\}$.}
\item{The forms $\{\w_{1}^{n-1},\dots, \w_{N}^{n-1}\}$ generate the space $\bigwedge^{2n-2}_{x}T^{*}M$ for all $x\in M$.}
\end{itemize}
\end{lemma}

\begin{proof} 
Define
$$G_{0}:= \{d_{v_{0}}g: T_{v_{0}}M\rightarrow T_{v_{0}}M\ |\ g\in G,\ g(v_{0}) = v_{0}\}$$
where $v_{0}$ is a fixed point in $M$. Since $\text{Ham}(M,\w)$ acts transitively on $M$ so the orbit $G_{0} . (\w_{|_{v_{0}}})$ is equal to the orbit $(G.\w)_{|_{v_{0}}}$, namely if $g\in G$, then there exist an element $h\in \text{Ham}(M,\w)\subset G$ so that $h(g(v_{0})) = v_{0}$. So we have $h\circ g \in G_{0}$ and $(h\circ g)^{*}\w_{|_{v_{0}}} = g^{*}\w_{|_{v_{0}}}$. If one proves that $\bigwedge^{2}T^{*}_{v_{0}}M = \big<G_{0}.(\w_{|_{v_{0}}})\big>$ then $\bigwedge^{2n-2}T^{*}_{v}M = \big<G_{0}.(\w^{n-1}_{|{v}})\big>$ for every $v\in U$ where $U$ is some neighborhood of $v_{0}$. Then by the action of $\text{Ham}(M,\w)$ we would move $U$ and cover $M$ with finitely many open sets for each of which there are finitely many diffeomorphisms in $G$ such that the corresponding pull backed forms form a large family on the corresponding open sets, hence all those diffeomorphims together would form a large family on $M$. Therefore it suffices to prove that $\bigwedge^{2}T^{*}_{v_{0}}M = \big<G_{0}.(\w_{|_{v_{0}}})\big>$. We know that there exist $\phi\in G$ so that $\phi^{*}\w\neq \pm\w$. (Composing with a map from $\text{Ham}(M,\w)$ if necessary) we have $\phi(v_{0}) = v_{0}$ and $(\phi^{*}\w)_{_{v_{0}}} \not\equiv \lambda \w_{_{v_{0}}}$ for any $\lambda\in \R$. By Lemma \ref{span} the orbit of a $2$-form independent from $\w$ under $Sp(2n)$, the group of linear transformations that preserve $\w$, together with $\w$ itself generates $\bigwedge^{2}T^{*}_{v_{0}}M$ and this finishes the proof.
\end{proof}

Let $\{f_{1},\dots,f_{N}\}$ be a set of diffeomorphisms in $G$ for which the set $\{f_{i}^{*}\w\}_{i=1}^{N}$ form a large family. Let $\phi_{1},\dots,\phi_{N}\in G$  so that $f_{i} = \phi_{i}\circ\dots\circ \phi_{1}$, $i = 1,\dots,N$. Consider the following map,
\[D: \text{Symp}(M,\w)\times \dots \times \text{Symp}(M,\w)\rightarrow \text{Diff}(M,\w^{n})\]
\[
(x_{1},\dots,x_{N})\mapsto f_{N}^{-1}\circ x_{N}\phi_{N}\dots x_{1}\phi_{1}
\]
If $X=(X_{1},\dots,X_{N})\in T_{Id}\bigg(\text{Symp}(M,\w)^{\times N}\bigg)$ where $Id = id\times \dots id$, then $$L(X):= dD(X) = \sum_{i=1}^{N}Df_{i}^{-1}(X_{i})$$
If $\{f_{t}\}_{0\leq t\leq 1}\subset \text{Diff}_{0}^{e}(M,\w^{n})$ is a smooth family of maps with $f_{0}=id$ then letting $X_{t} = \partial_{t}f_{t}\circ f_{t}^{-1}$ we get $f_{i}^{*}\w$-exact vector fields $X_{i}$ for $i=1,\dots, N$ such that $X_{t} = \sum_{i=1}^{N}X_{i}$. So we get,
$$L(Df_{1}(X_{1}),\dots, Df_{N}(X_{N})) = X_{t}$$ which implies that $f_{t}\in Im(D_{|_{\text{Ham}(M,\w)^{\times N}}})\subset G$ for $0 \leq t\leq 1$. (Note that $Df_{i}(X_{i})$ is $\w$-exact for all $i$.) Therefore we have $\text{Diff}^{e}_{0}(M,\w^{n})\subset G$ and if $H^{1}(M,\R) = 0$ then $\text{Diff}_{0}(M,\w^{n}) = G$.
\end{proof}

\vspace{0.5cm}
\begin{center}
3. PROOFS
\end{center}

Let us first prove the main theorem using the Maximality Theorem proved above.

\begin{corollary}[of Theorem \ref{maximality theorem}]
Let $(M^{2n},\w)$ be a closed connected symplectic manifold with $n\geq 3$ and let $G_{k}$ be as before. Then we have $G_{k} = \text{Symp}_{0}(M,\w)$ for all $0 < k < n$.
\end{corollary}

\begin{proof}
It is obvious that $\text{Ham}(M,\w)\subset G_{k}$. Let $f \in G_{k}$. By definition there exist a homotopy $\{f_{t}\}_{0\leq t\leq 1}\subset \text{Diff}(M,\w^{k})$ with $f_{0} = id$ and $f_{1} = f$. Differentiating $f_{t}^{*}\w^{k} = \w^{k}$ we get that $\mathcal{L}_{X_{t}}\w^{k} = 0$ where $X_{t} = \partial_{t}f_{t}\circ f_{t}^{-1}$.

$$\implies \mathcal{L}_{X_{t}}\w \wedge \w^{k-1} = 0\implies \mathcal{L}_{X_{t}}\w \wedge \w^{n-1} = 0$$ 
$$\implies \mathcal{L}_{X_{t}}\w^{n} = 0$$
so $f_{t}$ is volume preserving for all $t$, in particular so is $f = f_{1}$. Thus by the Maximality Theorem either $\text{Diff}_{0}^{e}(M,\w^{n})\subset G_{k}$ or $G_{k} \subset \text{Symp}_{0}(M,\w)$, depending on whether there is an element $\psi\in G$ so that $\psi^{*}\w\not\equiv \lambda \w$. We shall prove that the case $\text{Diff}_{0}^{e}(M,\w^{n})\subset G_{k}$ does not happen by constructing a diffeomorphism in $\text{Diff}_{0}^{e}(M,\w^{n})$ that does not preserve $\w^{k}$ for $0 < k < n$.
Let $U$ be a Darboux chart with local coordinate functions $(x_{1},\dots,x_{n},y_{1},\dots,y_{n})$. Define $f^{a}:\R^{2n}\rightarrow \R^{2n}$ by
$$(x,y)\mapsto (ax_{1},\dots,ax_{n},ay_{1},\dots,ay_{n-1},a^{-2n + 1}y_{n})$$
where $a > 1$ is a fixed real number.

Considering the path $\{f_{t}:= f^{a^{t}}\}_{0\leq t\leq 1}$ which is exact since $\R^{2n}$ is a contractible space we have that $f\in \text{Diff}_{0}(\R^{2n},\w^{n})$. So there exist a $(2n-2)$-form $\eta$ s.t. $$\iota_{X}\w\wedge \w^{n-1} = d\eta$$ 
where $X = {\partial_{t}}_{|_{t=0}}f_{t}$. Let $\phi: U\rightarrow \R$ be a smooth cut off function such that $\phi_{|_{\overline{B_{\delta}}}}\equiv 1$, $\phi_{|_{B_{2\delta}^{c}}}\equiv 0$ for a small enough $\delta > 0$. Consider the unique vector field $Y$ on $U$ defined by $\iota_{Y}\w\wedge \w^{n-1} = d(\phi\eta)$. Let $g_{t}$ be the flow of $Y$ extended by identity to the entire $M$. Then $g_{t}$ does not preserve $\w^{k}$ for all $0 < k < n$ and $t> 0$, since near the origin $g_{t}\equiv f_{t}$ and,
$$\w^{n-1}_{|_{0}}(e_{1},e_{1}',\dots,e_{n-1},e_{n-1}') = (-1)^{n}(n-1)!$$
$$f_{t}^{*}\w^{n-1}_{|_{0}}(e_{1},e_{1}',\dots,e_{n-1},e_{n-1}') = (-1)^{n}(n-1)!a^{t(2n-2)}$$
where $e_{i} = \frac{\partial}{\partial x_{i}},e_{i}' = \frac{\partial}{\partial y_{i}}$. So $g_{t}$ for all $t$ do not preserve the form $\w^{k}$ for $0 < k < n$ but is volume preserving.
\end{proof}

Using the Maximality theorem we proved the main result for closed symplectic manifolds, but in fact it holds for all symplectic manifolds as we will prove it now using some elementary methods. 

\begin{theorem}[Main theorem]

Let $(M^{2n},\w)$ be a symplectic manifold and the groups $G_{k},\ k=1,\dots,n$ defined as before. Then for $n > 2$ and $0 < k < n$ we have $G_{k} = G_{1}$.
\end{theorem}

\begin{proof}
Let $f\in G_{k}$. By definition there exist a smooth family of diffeomorphisms $f_{t}\in\text{Diff}_{0}(M,\w^{k}),\ t\in [0,1],$ so that $f_{0} = id$ and $f_{1} = f$. So for all $t$ we have $f_{t}^{*}\w^{k} = \w^{k}$. Differentiating the equation we get that,
$$0 = \partial_{t}f_{t}^{*}\w^{k} = f_{t}^{*}\mathcal{L}_{X_{t}}\w^{k}$$
$$\implies \mathcal{L}_{X_{t}}\w \wedge \w^{k-1} = 0.$$
To prove the theorem it is enough to prove that for $n > \max\{k,2\}$ the following map is injective,
$$\Omega^{2}(M)\overset{\w^{k-1}\wedge}\longrightarrow \Omega^{2k}(M).$$
This is a linear algebra problem. The following lemma will finish the proof.
\end{proof}

\begin{lemma}
\label{lemma}
Let $(V,\w)$ be a $2n$-dimensional real symplectic vector space, $n\geq 3$. Then the map $L:\bigwedge^{2}V\rightarrow \bigwedge^{2k}V$ defined by $\alpha\mapsto \w^{k-1}\wedge \alpha$ is injective for $0 < k < n$. 
\end{lemma}

\begin{proof}
We will be presenting two approaches. The first approach is a simple induction on the dimension of the vector space which we shall do it here and the second approach is presented in the next section on K\"ahler identities. Let $\{x_{1},\dots,x_{n},y_{1},\dots,y_{n}\}$ be the standard Darboux coordinates and $\alpha \in \bigwedge^{2}V$. Suppose $\w^{k-1} \wedge \alpha = 0$ for a fixed $0 < k < n$. We prove the statement by induction on $n$. First let $n = 3$. Then $k = 1,2$. For $k = 1$ there is nothing to prove. For $k = 2$ we have
$$\big(\sum_{i} dx_{i}\wedge dy_{i}\big)\wedge \big(\sum_{i < j} a_{ij}dx_{i}\wedge dx_{j} + \sum_{i,j} b_{ij}dx_{i}\wedge dy_{j} + \sum_{i < j}c_{ij}dy_{i}\wedge dy_{j}\big) = 0$$
where inside the first parentheses is $\w$ and the second parentheses is $\alpha$. There are three types of terms appearing in the wedge product. 
\begin{itemize}
\item{$D dx_{i}\wedge dx_{j}\wedge dx_{k}\wedge dy_{r}$: Such a term can only appear if $r\in \{i,j,k\}$. W.l.o.g assume $i = r$. The only possibility for generating such a term is taking $dx_{r}\wedge dy_{r}$ from the first parentheses and $dx_{j}\wedge dx_{k}$ from the second one. So $a_{jk} = 0$ for all $j < k$.}

\item{$D dx_{i}\wedge dy_{j}\wedge dy_{k}\wedge dy_{l}$: Here we get that $c_{ij} = 0$ for all $i < j$ by the same argument.}

\item{$D dx_{i}\wedge dx_{j}\wedge dy_{k}\wedge dy_{l}$: We must have $\{i,j\}\cap \{k,l\}\neq \emptyset$. W.l.o.g assume that $i = k$. If $j\neq l$ then $dx_{k}\wedge dx_{j}\wedge dy_{k}\wedge dy_{l}$ can only happen once so $b_{jl} = 0$. So we may assume the term looks like $Ddx_{k}\wedge dy_{k}\wedge dx_{l}\wedge dy_{l}$. The coefficient of this term in the product is $b_{ll} + b_{kk}$ so $b := b_{ll} = -b_{kk}$ for all $k \neq l$.}
\end{itemize}

Therefore since $n = 3$, there are three sets of coordinates $\{x_{i},y_{i}\}$, $i=1,2,3$ thus we have $b_{11} = -b_{22} = -(-b_{33}) = -(-(-b_{11})) = -b_{11}$ and we get $b = b_{11} = 0$ which means $\alpha = 0$. Now let us assume the statement for $M = (\R^{2n-2},\w_{std})$. We shall prove it for $M = (\R^{2n},\w:= \w_{std})$ and $0 < k < n$ where $n \geq 4$. For $k =1$ the statement is obvious, so assume $k\geq 2$. We have 
$$\w^{k-1} = (-1)^{k}(k-1)!\sum_{1\leq i_{1}\leq \dots\leq i_{k-1}\leq n} dx_{i_{1}}\wedge \dots\wedge dx_{i_{k-1}}\wedge dy_{i_{1}}\wedge \dots \wedge dy_{i_{k-1}}$$
and let 
$$\alpha = \sum_{i < j} a_{ij}dx_{i}\wedge dx_{j} + \sum_{i,j} b_{ij}dx_{i}\wedge dy_{j} + \sum_{i < j}c_{ij}dy_{i}\wedge dy_{j}$$

For a set of indices $I = \{i_{1} < \dots < i_{k-1}\}$ denote $dx_{i_{1}}\wedge \dots\wedge dx_{i_{k-1}}$ by $dx_{I}$ and $dy_{i_{1}}\wedge \dots\wedge dy_{i_{k-1}}$ by $dy_{I}$. Let $i < j$ be two arbitrary indices. Note that since $k - 1 \leq n - 2$ there is a set of indices $I_{ij}$ with $\#I_{ij} = k-1$ and $i,j\notin I_{ij}$. So the following term appears in the product $\alpha\wedge \w^{k-1}$
$$a_{ij}dx_{i}\wedge dx_{j}\wedge dx_{I_{ij}}\wedge dy_{I_{ij}}.$$
(unless $a_{ij} = 0$ which we are aiming to prove.) But such a term can only appear once (with the coefficient $a_{ij}$) since for instance if $dx_{i}$ had come from the term $\w^{k-1}$ then $dy_{i}$ would have appeared too, so we have $a_{ij} = 0$. The same argument shows that $c_{ij} = 0$ for all $i < j$. So we have 
$$\alpha = \sum_{i,j} b_{ij}dx_{i}\wedge dy_{j}.$$
Now if $r,s$ are two distinct indices then if we consider the term 
$$b_{rs}dx_{r}\wedge dy_{s}\wedge dx_{I}\wedge dy_{I}$$
where $I$ is a set of indices so that $r,s\notin I$ then such a term also appears once because if for instance $J$ is a set of indices that includes $r$ then $dy_{r}$ will appear too in the product $b_{-s}dx_{-}\wedge dy_{s}\wedge dx_{J}\wedge dy_{J}$ which will give us a different form.

Thus we have reduced $\alpha$ to the following form
$$\alpha = \sum_{i}b_{ii}dx_{i}\wedge dy_{i}$$ 
Let $X:= \frac{\partial}{\partial x_{1}}$. Then we have 
$$0 = \iota_{X}(\alpha \wedge \w^{k-1}) = \iota_{X}\alpha \wedge \w^{k-1} + \alpha\wedge \iota_{X}\w^{k-1}$$
$$= b_{11}dy_{1}\wedge \w^{k-1} + (k-1)\alpha \wedge \iota_{X}\w \wedge \w^{k-2}$$
$$= b_{11}dy_{1}\wedge \w^{k-1} + (k-1)\ dy_{1}\wedge\alpha \wedge \w^{k-2}$$
$$= dy_{1}\wedge\big(b_{11}\w + (k-1)\alpha\big)\wedge \w^{k-2}$$
If we write $\w^{k-2} = dy_{1}\wedge (\dots) + \w_{1}$ and $b_{11}\w + (k-1)\alpha = dy_{1}\wedge(\dots) + \alpha_{1}$ where $\w_{1}$ and $\alpha_{1}$ have no $dy_{1}$ factor then we shall have 
$$\alpha_{1}\wedge \w_{1} = 0.$$
Note that here $\alpha_{1}$ is a two form that consists of factors $dx_{2},\dots, dx_{n},dy_{2},\dots, dy_{n}$ and it is also clear that
$$\w_{1} = (\sum_{i=2}^{n}dx_{i}\wedge dy_{i})^{\wedge (k-2)}.$$
Since $k -1 < n - 1$ and $n-1 \geq 3$ so by induction hypothesis we have $\alpha_{1} = 0$. This gives us the following equality,
$$b_{11} = -(k-1)b_{ii} \ \ \ \ \forall i\neq 1.$$
If we had contracted the form $\alpha \wedge \w^{k-1}$ by $\frac{\partial}{\partial x_{i}}$ instead of $X = \frac{\partial}{\partial x_{1}}$ we would have got the following through the exact same argument 
$$b_{ii} = -(k-1)b_{jj} \ \ \ \ \forall j\neq i$$

And finally since $n\geq 3$,
$$b_{11} = -(k-1)b_{22} = (k-1)^{2}b_{33} = -(k-1)^{3}b_{11}$$
so $b_{11} = 0$ and therefore $b_{ii} = 0$ for all $i$ and consequently $\alpha = 0$.
\end{proof}

To present a second proof of Lemma \ref{lemma} we start by recalling few definitions.

\begin{definition}
Let $V$ be a finite dimensional real vector space. An automorphism $J:V\rightarrow V$ is called an almost complex structure on $V$ if $J\circ J = -id$. 
\end{definition}

\begin{definition}
An almost complex structure on an finite dimensional inner product space $(V,g)$ is called $g$-compatible if it is an isometry with respect to $g$, i.e.
$$g(J\cdot,J\cdot) = g(\cdot,\cdot).$$
\end{definition}
\begin{definition}
Let $(V^{2n},g,J)$ be a finite dimensional real vector space equipped with an inner product $g$ and a $g$-compatible almost complex structure $J$. Define $L,\Lambda, H\in End(\bigwedge^{*}V^{*})$ as follows
\begin{itemize}
\item{$L(\alpha):= \w \wedge \alpha$ where $\w := g(J\cdot,\cdot)$ and $\alpha \in \bigwedge^{*}V^{*}$}
\item{$\Lambda$ is the adjoint of $L$ with respect to $g$}
\item{$H_{|_{\bigwedge^{k}}}:= (k-n)id$ for all $k\geq 0$.}
\end{itemize}
\end{definition}

\begin{definition}
Let $(V^{2n},g,J)$ be a finite dimensional real vector space as before. Define $vol\in \bigwedge^{2n}V^{*}$ to be a volume form that defines the same orientation as does $J$, and it evaluates $1$ on the orthonormal oriented basis w.r.t $g$. Then define the Hodge $*$-operator $*: \bigwedge^{k}V^{*}\rightarrow \bigwedge^{2n-k}V^{*}$ by the equation $\alpha\wedge *\beta = g(\alpha, \beta)vol$, $\forall\ \alpha\in \bigwedge^{k}V^{*}$. Here $g$ is an induced inner product on the space of higher exterior products which is denoted by the same letter $g$.
\end{definition}

\begin{theorem}\cite[Prop.1.2.30]{complexgeometry}
\label{Kahler identities}
Let $(V,g,J)$ be a finite dimensional real vector space equipped with an inner product $g$ and a $g$-compatible almost complex structure $J$. Then the following holds,

\begin{enumerate}
\item $(*\circ *)_{|_{\bigwedge^{k}}} = (-1)^{k}$.
\item $\Lambda = *^{-1}\circ L\circ *$.
\item $[H,L] = 2L,\ \ \ [H,\Lambda] = -2\Lambda, \ \ \ [L,\Lambda] = H$.
\item $[L^{i},\Lambda](\alpha) = i(k - n + i - 1)L^{i-1}(\alpha)\ \ \forall \alpha \in {\bigwedge}^{k}V^{*}$.
\item There is a direct decomposition of the form 
$${\bigwedge}^{k}V^{*} = \oplus_{i\geq 0}L^{i}(P^{k - 2i})$$
where $P^{k}:= \ker(\Lambda)\cap\bigwedge^{k}V^{*}$.
\item for $k > n$ we have $P^{k} = 0$
\item The map $L^{n-k}: P^{k}\rightarrow \bigwedge^{2n - k}V^{*}$ is injective for $k \leq n$.
\item The map $L^{n-k}: \bigwedge^{k}V^{*}\rightarrow \bigwedge^{2n - k}V^{*}$ is bijective for $k \leq n$.
\label{lastpart}
\end{enumerate}
\end{theorem}

\begin{proof}
We shall only prove the last two parts. To prove the part $(7)$, let $0\neq \alpha \in P^{k}$. Let $i > 0$ be the smallest integer for which $L^{i}(\alpha) = 0$. Then we have $0 = [L^{i},\Lambda](\alpha) = i(k-n + i -1)L^{i-1}(\alpha)$. So we should have $k-n + i -1 = 0$ which means $L^{n-k}(\alpha) = L^{i-1}(\alpha) \neq 0$. To prove the last part let $0\neq \alpha \in \bigwedge^{k}V^{*}$. By part $(5)$ we can write $\alpha = \oplus_{i\geq 0}\alpha_{i}$ where $\alpha_{i} = L^{i}(\beta_{i})$ for some $\beta_{i}\in P^{k - 2i}$. So we have $L^{n-k}(\alpha) = \oplus_{i\geq 0}L^{n-k + i}(\beta_{i})$. But since $\alpha \neq 0$ so there must be a $j$ such that $\beta_{j}\neq 0$. Then by part $(7)$ we have $L^{n-k + 2j}(\beta_{j})\neq 0$ hence $L^{n-k+j}(\beta_{j})\neq 0$ which implies $L^{n-k}(\alpha)\neq 0$.
\end{proof}

\begin{proof}[Proof of Lemma \ref{lemma}]
Letting $k = 2$ in part \ref{lastpart} of Theorem \ref{Kahler identities} proves the Lemma.
\end{proof}

\newpage
\bibliographystyle{unsrt}
\bibliography{bib.bib} 
\nocite{*}

\end{document}